\documentclass{amsart}
\usepackage{amssymb,latexsym}
\newcommand{\PP}{{\bf P}}
\theoremstyle{plain}

\newtheorem{corollary}{Corollary}
\newtheorem{proposition}{Proposition}

\theoremstyle{definition}

\begin{document}
\title[On the rank of a symmetric form]{On the rank of a symmetric form}
\author[Kristian Ranestad]{Kristian Ranestad}
\address{Matematisk institutt\\
         Universitetet i Oslo\\
         PO Box 1053, Blindern\\
         NO-0316 Oslo\\
         Norway}
\email{ranestad@math.uio.no}
\urladdr{http://folk.uio.no/ranestad/}
\thanks{Both authors were supported by Institut
Mittag-Leffler, and thank Jaroslaw Buczynski for clearifying 
    discussions on various notions of rank.}
\author[Frank-Olaf Schreyer]{Frank-Olaf Schreyer}
\address{Mathematik und Informatik\\
Universit\"at des Saarlandes\\
 Campus E2 4\\
 D-66123 Saarbr\"ucken, Germany}
\email{schreyer@math.uni-sb.de } 
\urladdr{http://www.math.uni-sb.de/ag/schreyer/}

\date{\today}

\date{\today}
\begin{abstract}
We give a lower bound for the degree of a finite apolar subscheme of a symmetric form $F$, in 
terms of the degrees of the generators of the annihilator ideal 
$F^{\bot}$.  In the special case, when $F$ is a monomial 
$x_{0}^{d_{0}}\cdot x_{2}^{d_{2}}\cdot\ldots\cdot x_{n}^{d_{n}}$ with 
$d_{0}\leq d_{1}\leq\ldots\leq d_{n-1}\leq d_{n}$ we 
deduce that the minimal length of an apolar subscheme of $F$ is $(d_{0}+1)\cdot\ldots\cdot 
(d_{n-1}+1)$, and if $d_{0}=\ldots=d_{n}$, then this minimal length 
coincides with the rank of $F$.
\end{abstract}

\maketitle


Let $F\in T= K[x_0,\ldots ,x_n]$ be a homogeneous form and 
let $S= K[y_0,\ldots ,y_n]$ be the ring of commuting differential 
operators acting on $T$. The action is called apolarity, and defines 
$S$ as a natural coordinate ring on the projective space $\PP(T_{1})$ 
of $1$-dimensional subspaces of $T_{1}$.
 The annihilator of $F$ is an ideal $F^{\bot}\subset 
S$. A finite subscheme $\Gamma\subset \PP(T_{1})$ is apolar to $F$ if 
the homogeneous ideal $I_{\Gamma}\subset S$ is contained in 
$F^{\bot}$.  

We define the cactus rank $cr(F)$ as
$$cr(F)={\rm min}\{{\rm deg}\Gamma | \Gamma\subset \PP(T_{1}), 
{\rm dim}\Gamma=0, I_{\Gamma}\subset F^{\bot}\},$$ 
the smoothable rank $sr(F)$  as
$$sr(F)={\rm min}\{{\rm deg}\Gamma | \Gamma\subset \PP(T_{1}) 
\;\;{\rm smoothable}, 
{\rm dim}\Gamma=0, I_{\Gamma}\subset F^{\bot}\}$$ 
and the rank $r(F)$ as
$$r(F)={\rm min}\{{\rm deg}\Gamma | \Gamma\subset \PP(T_{1}) \;\;{\rm smooth}, 
{\rm dim}\Gamma=0, I_{\Gamma}\subset F^{\bot}\}.$$
Clearly $cr(F)\leq sr(F)\leq r(F)$.
We shall give lower bounds for these ranks in terms of the generators 
of the ideal $F^{\bot}$. The related notion of 
border rank, $br(F),$ is defined as the minimal $k$ such that $[F]$ lies in 
the Zariski closure of the set of forms of rank $k$ in $\PP(T_{{\rm 
deg}F})$.  
  In general $br(F)\leq sr(F)$, and strict inequality occurs, so our lower bounds for $sr(F)$ does 
not apply unconditionally to $br(F)$.  
For applications of these notions of rank to powersum decompositions of symmetric forms and to equations of secant 
varieties, see \cite{RS},  \cite{LT} and \cite{BB}.

We define the degree of $F^{\bot}$ to be the length of the quotient 
algebra $S_{F}=S/F^{\bot}$.  

\begin{proposition} If the ideal of $F^{\bot}$ is generated in degree 
    $d$ and  $\Gamma\subset \PP(T_{1})$ is a finite apolar subscheme 
    to $F$, then $${\rm deg}\; \Gamma\geq \frac{1}{d}\;{\rm 
    deg}F^{\bot}.$$
    \end{proposition}
    \begin{proof}
	Taking cones, we may assume that $F^{\bot}$ and $I_{\Gamma}$ 
	define subschemes $X$ and $Y$ of pure dimension $r$ and $r+1$ in $\PP^{N}$.
	Furthermore ${\rm deg}\; Y={\rm deg}\; \Gamma$ and ${\rm deg}\; 
	X={\rm deg}\; F^{\bot}$.
	The apolarity condition says that $I_{X}\supset I_{Y}$, i.e. 
	that $X\subset Y$ as schemes.  Now, take an element $g$ in 
	$I_{X}$ that does not contain any component of $Y$.  Then the 
	hypersurface $G=\{g=0\}$ has proper intersection with $Y$ and 
	contains $X$.  Therefore, by Bezout, 
	$${\rm deg}G\cdot {\rm deg}Y\geq {\rm deg} X.$$
	The proposition follows by taking $g$ of degree $d$.
	\end{proof}
	
	\begin{corollary} If the ideal of $F^{\bot}$ is generated in degree 
    $d$, then the cactus rank $$cr(F)\geq \frac{1}{d}\;{\rm 
    deg}F^{\bot}.$$
    \end{corollary}

	\begin{corollary} If $F$ is a monomial, $F=x_{0}^{d_{0}}\cdot 
	    x_{1}^{d_{1}}\cdot\ldots\cdot x_{n}^{d_{n}}$ with 
$d_{0}\leq d_{1}\leq\ldots\leq d_{n}$, then  the cactus rank and the 
smoothable rank coincide and equals $$cr(F)=sr(F)=(d_{0}+1)\cdot\ldots\cdot 
(d_{n-1}+1).$$  If furthermore $d_{n}=d_{0}=d$, i.e. $F=(x_{0}\cdot 
x_{1}\cdot\ldots\cdot x_{n})^{d}$, then 
$r(F)=cr(F)=sr(F)=(d+1)^{n}$.\end{corollary}
\begin{proof}  When $F=x_{0}^{d_{0}}\cdot 
    x_{1}^{d_{1}}\cdot\ldots\cdot x_{n}^{d_{n}}$, then $F^{\bot}$ is 
    the complete intersection generated by the forms
    $$y_{0}^{d_{0}+1}, y_{1}^{d_{1}+1}, \ldots, y_{n}^{d_{n}+1}.$$
    So it is generated in degree $d_{n}+1$, while $F^{\bot}$ has 
    degree $$(d_{0}+1)\cdot (d_{1}+1)\cdot \ldots\cdot  (d_{n}+1).$$
    The formula for the cactus rank follows, since the first $n$ 
    generators define a finite apolar subscheme of degree 
    $(d_{0}+1)\cdot \ldots\cdot  (d_{n-1}+1).$
  Now, any complete intersection is smoothable, so the smoothable rank 
  equals the cactus rank for $F$.
    If $d_{0}=d_{n}=d$, then the forms of degree $d+1$ in 
    $F^{\bot}$ has no basepoints so, by Bertini, $n$ general forms in 
     $F^{\bot}$ of degree $(d+1)$ define a smooth finite subscheme  
     of degree $(d+1)^{n}$ in 
    $\PP(T_{1})$.\end{proof}
    
    \

\end{document}